\documentclass[12pt]{article}
\usepackage{etex}
\usepackage[margin=1in]{geometry}

\usepackage{tikz}
\usetikzlibrary{decorations.markings}
\tikzset{middlearrow/.style n args={4}{
	decoration={
  		markings,
  			mark=at position #1 with {\arrow{#2},\node[transform shape,#4] {#3};}},postaction={decorate}},
  		middlearrow/.default={.5}{>}{}{below}	}

\usepackage{amsmath, amssymb, amsfonts, amsthm}
\usepackage{mathtools}
\usepackage{thmtools,thm-restate}

\usepackage{parskip}
\makeatletter
\def\thm@space@setup{%
	\thm@preskip=\abovedisplayskip \thm@postskip=.5\thm@preskip
}
\makeatother

\theoremstyle{plain}
\newtheorem{thm}{Theorem}[section]
\newtheorem{prop}[thm]{Proposition}
\newtheorem{cor}[thm]{Corollary}
\newtheorem{lem}[thm]{Lemma}

\theoremstyle{definition}
\newtheorem{defn}[thm]{Definition}
\newtheorem{rem}[thm]{Remark}

\newtheorem*{ack}{Acknowledgements}

\usepackage[all]{xy}

\usepackage{subcaption}

\usepackage{enumitem}
\setlist[description]{
	font = \normalfont}
\setlist[enumerate]{label=(\arabic*)}

\usepackage[british]{isodate}

\usepackage{hyperref}

\DeclareMathOperator{\Axis}{Axis}
\newcommand{\Z}{\mathbb{Z}}
\newcommand{\R}{\mathbb{R}}
\DeclareMathOperator{\Aut}{Aut}
\DeclareMathOperator{\Out}{Out}
\DeclareMathOperator{\Isom}{Isom}
\DeclareMathOperator{\FR}{FR}
\DeclareMathOperator{\Fact}{Fact}
\DeclareMathOperator{\Perm}{Perm}
\DeclareMathOperator{\Fix}{Fix}
\DeclareMathOperator{\SL}{SL}

\title{Serre's Property (FA) for automorphism groups of free products}
\author{Naomi Andrew\thanks{School of Mathematical Sciences, University of Southampton, Southampton, SO17 1BJ, UK \newline
Email: \texttt{N.G.Andrew@soton.ac.uk}}}
\date{}
\begin{document}

\maketitle

\begin{abstract}
We provide some necessary and some sufficient conditions for the automorphism group of a free product of (freely indecomposable, not infinite cyclic) groups to have Property (FA). The additional sufficient conditions are all met by finite groups, and so this case is fully characterised. Therefore this paper generalises the work of Leder in \cite{Leder2018FA} for finite cyclic groups, as well as resolving the open case of that paper.
\end{abstract}

\section{Introduction}

Serre introduced Property (FA) in \cite{Serre2003} as a `near opposite' to a group splitting as a free product with amalgamation or an HNN extension. A group $G$ has Property (FA) if every action of $G$ on a tree has a fixed point.

Serre proves (as Theorem 15, p58 of \cite{Serre2003}) that Property (FA) is equivalent to the following conditions \begin{enumerate}
\item $G$ is not a (non-trivial) amalgamated free product
\item $G$ has no quotient isomorphic to $\Z$
\item $G$ is not the union of a strictly increasing sequence of subgroups.
\end{enumerate}

If $G$ is countable, then the third condition is equivalent to finite generation; there are uncountable groups satisfying Property (FA) (\cite{KoppelbergTits1974}). Examples of groups with Property (FA) include finitely generated torsion groups and $\SL(n,\Z)$ for $n\geq 3$ (both due to Serre in \cite{Serre2003}); $\Aut(F_n)$ for $n \geq 3$ (due to Bogopolski in \cite{BogopolskiAutFnFA}, with an alternative proof in \cite{CullerVogtmannFA}) and the automorphism group of a free product of at least four copies of $\Z/n\Z$ (due to Leder in \cite{Leder2018FA}).

In fact, Leder shows (in most cases) that for free products of finite cyclic groups, whether the automorphism group has Property (FA) depends only on the number of times each isomorphism class appears. Our results give the following generalisation and completion of Leder's work:

\begin{cor}
	\label{cor:finite_gps}
	Let $G$ be a free product of finite groups. Then $\Aut(G)$ has Property (FA) if and only if all but possibly one factor appear at least four times (up to isomorphism), and the remaining factor (if present) appears only once.
\end{cor}

This is a consequence of our main results which are, in the positive direction:

\begin{restatable*}{thm}{sufficient}
\label{thm:Aut(G)_has_FA}
Let $G$ be a (finite) free product of groups such that either \begin{enumerate}[label=(\arabic*)]
\item each free factor has Property (FA), its automorphism group has finite abelianisation and cannot be expressed as the union of a properly increasing sequence of subgroups, and (up to isomorphism) appears at least four times in the decomposition; or
\item There is a free factor appearing exactly once that has Property (FA) and its automorphism group has Property (FA); and all other free factors are as in (1).
\end{enumerate}
Then $\Aut(G)$ has Property (FA).
\end{restatable*}

And in the opposite direction:
\begin{restatable*}{thm}{necessary}
\label{thm:has_not_FA}
Let $G$ be a free product of (freely indecomposable) groups, with no infinite cyclic factors. If \begin{enumerate}[label=(\arabic*)]
\item any free factor appears exactly two or three times, or any two free factors appear exactly once; or
\item the automorphism group of any factor appearing exactly once does not have Property (FA); or
\item the automorphism group of any factor appearing more than once does not have finite abelianisation or can be expressed as a union of a properly increasing sequence of subgroups.
\end{enumerate}
Then $\Aut(G)$ does not have Property (FA).
\end{restatable*}

These imply Corollary \ref{cor:finite_gps} since finite groups and their automorphism groups have Property (FA), and so the extra conditions of Theorem \ref{thm:Aut(G)_has_FA} are always satisfied.

Comparing Theorems \ref{thm:Aut(G)_has_FA} and \ref{thm:has_not_FA}, most of the sufficient conditions in Theorem \ref{thm:Aut(G)_has_FA} are also necessary. The exception is the requirement that each factor has Property (FA): since the structure of the automorphism group places significant restrictions on the possible trees it could act on, it seems plausible that there are examples of groups that act on trees but not in a way that extends to the automorphism group of their free product.

The cases with infinite cyclic factors are in general still open although some cases of Theorem \ref{thm:has_not_FA} go through allowing free rank $1$ or $2$, and (as observed above) in the opposite direction $\Aut(F_n)$ has Property (FA) for $n \geq 3$.

All of the groups considered have finite index subgroups that do act on trees, which will be shown as Proposition \ref{prop:no_prop_T}, and so we obtain

\begin{restatable*}{cor}{propertyT}
Suppose $G$ is a (finite, non-trivial) free product where each factor is freely indecomposable and not infinite cyclic. Then $\Aut(G)$ does not have Kazhdan's Property (T).
\end{restatable*}

\begin{rem}
In view of Remark 1.10 of \cite{CornulierKar2011Wreathprods}, Theorem \ref{thm:Aut(G)_has_FA}(1) is true for Property (F$\R$), as is Theorem \ref{thm:Aut(G)_has_FA}(2) with the extra hypothesis that the free factor appearing once only is finitely generated.
\end{rem}

\begin{ack}
I am grateful to my supervisor, Armando Martino, for all his guidance and encouragement. I am also grateful to Ric Wade for pointing out the obsevation of Proposition \ref{prop:no_prop_T}, as well as to Ashot Minasyan for helpful comments on this manuscript.
\end{ack}

\section{Background}

\subsection{Actions on trees}

First, we collect some lemmas about trees, subtrees and fixed point sets of elliptic subgroups of groups acting on trees, that are needed at various points in the later arguments. Many of the statements and proofs hold for both real and simplicial trees, but unless otherwise specified, all trees are simplicial trees equipped with the edge-path metric.

\begin{lem}
\label{lem:naomi_intersection}
Let $X_i$ be a family of subtrees of a tree $T$ with non-empty intersection, and let $Y$ be another subtree. Suppose that for each $i$, $X_i \cap Y$ is non-empty. Then $(\bigcap X_i) \cap Y$ is also non-empty.
\end{lem}
\begin{proof}
Let $v$ be a nearest point in $Y$ to $\bigcap X_i$. Then for each $i$, $X_i$ contains $v$, since $X_i$ includes both $\bigcap X_i$ and part of $Y$. So $v \in \bigcap X_i$, and since it was in $Y$ by definition it is in $(\bigcap X_i) \cap Y)$.
\end{proof}

In the finite case, but not in general, we may weaken the hypotheses to give the following lemma. (In fact, it can be proved by using Lemma \ref{lem:naomi_intersection} as an induction step.)

\begin{lem}[Serre, Lemma 10 of \cite{Serre2003}]
\label{lem:serre_intersection}
Let $X_1,\dots, X_m$ be subtrees of a tree $T$. If the $X_i$ meet pairwise, then their intersection is non-empty.
\end{lem}

\begin{lem}
\label{lem:commuting_fixed_pt}
If two elliptic subgroups $H,K$ commute, then the subgroup they generate (isomorphic to $H \times K$ if their intersection is trivial) is elliptic.
\end{lem} 
\begin{proof}
Consider some point $v$ in $\Fix(H)$. Since $vkh=vhk=vk$, for all $h$ and $k$, the point $vk$ is also in $\Fix(H)$. So the geodesic $[v,vk]$ is contained in $\Fix(H)$. Since $k$ is elliptic, the midpoint of this geodesic is fixed by $k$ which puts it in the intersection $\Fix(H) \cap \Fix(k)$ which must be non-empty. Since $\Fix(K)$ is the non-empty intersection of all the $\Fix(k)$, the subtrees $\Fix(k)$ and $\Fix(H)$ satisfy Lemma \ref{lem:naomi_intersection}, and so $\Fix(H) \cap \Fix(K)$ is non-empty.
\end{proof}

Note that this also gives that the direct product of two groups with Property (FA) itself has Property (FA). The converse is also true, since factors are quotients so if either factor has an action on a tree the direct product will.

\begin{lem}\hfill
\label{lem:fixed_pt_new}
\begin{enumerate}
\item Suppose a tree has subtrees $S_1,S_2,T_1,T_2$ such that $S_1$ has non-empty intersection with $T_1$ and $T_2$, and $S_2$ has non-empty intersection with $T_1$ and $T_2$. Then $S_1$ and $S_2$ have non-empty intersection, or $T_1$ and $T_2$ have non-empty intersection.
\item Suppose a group $G$ acts on a tree, and has subgroups $H_1$ and $H_2$ which are elliptic, and an element $g$ such that $H_1$ has common fixed points with $H_1^g$ and $H_2^g$, and $H_2$ has common fixed points with $H_1^g$ and $H_2^g$. Then $H_1$ and $H_2$ have a common fixed point.
\end{enumerate}
\end{lem}
\begin{proof} \hfill
\begin{enumerate}
\item Suppose $S_1$ and $S_2$ do not intersect. Consider the bridge joining $S_1$ and $S_2$. Since $T_1$ has non-empty intersection with both these subtrees, $T_1$ contains this bridge. Similarly, $T_2$ contains this bridge. So $T_1 \cap T_2$ contains the bridge, and must be non-empty.
\item The fixed point subtrees of the four subgroups satisfy the conditions of part (1), so either $H_1$ and $H_2$ or $H_1^g$ and $H_2^g$ have a common fixed point. But since $\Fix(H_1^g) \cap \Fix(H_2^g) = (\Fix(H_1)\cap \Fix(H_2))g$ if one is non-empty both are. So in fact both are non-empty and so $H_1$ and $H_2$ have a common fixed point. \qedhere
\end{enumerate}
\end{proof}

\subsection{Automorphisms of free products}

Presentations of the automorphism group of a free product were found by Fouxe-Rabinovitch in \cite{FouxeRabinovitch1940AutFreeProd} and \cite{FouxeRabinovitch1941AutFreeProd} and later by Gilbert in \cite{Gilbert1987AutFreeProd}. (Gilbert's is a finite presentation, under some reasonable finiteness assumptions on the factor groups and their automorphisms.) They assume that the free product is given as a Grushko decomposition: 

\begin{thm}[Grushko decomposition]
	Any finitely generated group $G$ can be decomposed as a free product $G=G_1*\dots G_k*F_r$, where the $G_i$ are non-trivial, freely indecomposable and not infinite cyclic, and $F_r$ is a free group of rank $r$.
	Further, the $G_i$ are unique up to conjugacy, and the rank of $F_r$ is unique.
\end{thm}

They distinguish three kinds of automorphism, which generate the whole automorphism group:
\begin{itemize}
\item Factor automorphisms, which are automorphisms of just one free factor and don't affect the rest;
\item Permutation automorphisms, which permute isomorphic free factors according to a fixed, compatible set of isomorphisms;
\item Whitehead automorphisms, of two kinds: partial conjugations sending a free factor $G_i$ to $G_i^a$ and, if $G_i$ is an infinite cyclic factor, transvections sending $G_i$ to $aG_i$. In both cases $a$ is required to be an element of a different free factor.
\end{itemize}

Gilbert gives the following characterisation of subgroups generated by the factor and permutation automorphisms (denoted $\Fact(G)$ and $\Perm(G)$ respectively):
\begin{prop}[Gilbert, Proposition 3.1 of \cite{Gilbert1987AutFreeProd}]
\label{prop:structure_of_factor_perm_auts}
Let $G=G_1 \ast \dots G_k \ast F_r$ and suppose (after reordering if necessary) that $G_1, \dots G_d$ are representatives of (all the) distinct isomorphism classes of the $G_i$, and that the isomorphism class represented by $G_i$ occurs $n_i$ times. Also suppose that we take a fixed splitting of $F_r$ as a free product of infinite cyclic groups. Then: \begin{enumerate}
\item $\Fact(G)\cong \prod_{i=1}^d \Aut(G_i) \times C_2^r$
\item $\Perm(G) \cong \prod_{i=1}^d S_{n_i} \times S_r$
\item $\langle \Fact(G),\Perm(G) \rangle \cong \prod_{i=1}^d (\Aut(G_i) \wr S_{n_i}) \times (C_2 \wr S_r)$
\end{enumerate}
\end{prop}

(The wreath products are permutation wreath products on a set of $n$ elements, not $n!$.)

If $G$ has no infinite cyclic factors we can consider the subgroup generated by partial conjugations. Write $(A,b)$ for the automorphism that conjugates every element of $A$ by $b$. 
The subgroup generated by these is denoted $\FR(G)$ and has the following presentation given explicitly as Proposition 3.1 of \cite{CollinsGilbert1990AutFreeProd} (althought it can be deduced from   \cite{FouxeRabinovitch1940AutFreeProd} and \cite{Gilbert1987AutFreeProd}):

\begin{prop}
\label{prop:aut_G_pres}
Suppose $G$ is a (non-trivial) free product of freely indecomposable groups with no infinite cyclic factors. Then the subgroup $\FR(G)$ of $\Aut(G)$ is generated by the partial conjugations $(A,b)$ subject to the relations: \begin{align}
(A,b)(A,b')&=(A,b'b) \label{rel:mult_table}\\
(A,b)(C,d)&=(C,d)(A,b) &\text{for $A \neq C, b \notin C, d \notin A$} \label{rel:commuters}\\
[(A,b)(C,b),(A,c)]&=1 &\text{for $A,B,C$ all different, $b \in B$, $c \in C$}
\end{align}
If each factor is finitely generated or presented, the same is true of $\FR(G)$, by rewriting the generators (eg) $(A,b)$ in terms of a (finite) generating set for the subgroup $B$, and eliminating all unnecessary relators.

Finally, a presentation for $\Aut(G)$ is found by adding a set of generators and relations for $\langle\Fact(G),\Perm(G) \rangle$ together with the relations $\varphi^{-1} (A,b) \varphi = (A\varphi,b\varphi)$ for each $\varphi \in \langle \Fact(G),\Perm(G) \rangle$.
\end{prop}

Since inner factor automorphisms were excluded from $\FR(G)$, it has trivial intersection with $\langle\Fact(G),\Perm(G)\rangle$. So together with the final relation above giving that it is normal, we have that $\Aut(G) = \FR(G) \rtimes \langle \Fact(G),\Perm(G) \rangle$.


\subsection{Characteristic subgroups}

In order to extend the actions constructed in Theorem \ref{thm:has_not_FA} from one or two isomophism classes of factors to the whole group, we need to have access to certain quotients of $\Aut(G)$. These are given by considering characteristic subgroups of $G$.
\begin{defn}
A subgroup $H$ of a group $G$ is characteristic if $H\varphi = H$ for all automorphisms $\varphi$ of $G$.
\end{defn}

Characteristic subgroups include the commutator subgroup, the subgroup generated by all finite order elements, and the normal subgroup generated by all free factors of the same isomorphism class (once the group is written as a Grushko decomposition, and provided the isomorphism class is not $\Z$).

\begin{prop}
If $N$ is a characteristic subgroup of $G$, then there is a homomorphism $\Aut(G) \to \Aut (G/N)$ given by $\varphi \mapsto (Ng \mapsto N(g\varphi))$.
\end{prop}

If $G$ is a free product and $N$ the normal subgroup generated by all free factors in a given isomorphism class this map is onto: the quotient group is the free product of the remaining factors, and all the generators involving only those factors are mapped to `themselves'. 

\subsection{Bass-Serre theory}

We restate some of the definitions and results of Bass-Serre theory, making sure the notation lines up with this paper. In particular, the action will be on the right. (This is closest to the exposition by Bass \cite{Bass1993}, but other expositions can be found in \cite{Serre2003} and \cite{DicksDunwoody}.)

\begin{defn}
	A \emph{graph of groups} $\mathcal{G}$ consists of a graph $\Gamma$  together with groups $G_v$ for every vertex and $G_e = G_{\overline{e}}$ for every (oriented) edge, and monomorphisms $\alpha_e:G_e \to G_{\tau(e)}$ for every (oriented) edge.
\end{defn}

(Here, the graph $\Gamma$ should be understood as it is defined by Serre, with edges in oriented pairs indicated by $\overline{e}$, and maps $\iota(e)$ and $\tau(e)$ from each edge to its initial and terminal vertices.)

The fundamental group of a graph of groups can be defined in two ways, with respect to a maximum tree of the graph, and by considering loops in the graph of groups. We take the second route, which simplifies some subsequent calculations.

\begin{defn}[Paths]
Let $F(\mathcal{G})$ be the group generated by all the vertex groups and all the edges of $\mathcal{G}$, subject to relations $e\alpha_e(g)\overline{e}=\alpha_{\overline{e}}(g)$ for $g \in G_e$. Note that taking $g=1$ this gives that $e^{-1}=\overline{e}$, as expected.

Define a \emph{path} (of length $n$) in $F(\mathcal{G})$ to be a sequence $g_0e_1g_1 \dots e_ng_n$, where each $e_i$ has $\iota(e_i)=v_{i-1}$ and $\tau(e_i)=v_i$ for some vertices $v_i$ (so there is a path in the graph), and each $g_i \in G_{v_i}$. A \emph{loop} is a path where $v_0=v_n$.
\end{defn}

The set of all paths in $F(\mathcal{G})$ forms a groupoid (sometimes called the fundamental groupoid of $\mathcal{G}$).

\begin{defn}[Reduced paths]
A path is \emph{reduced} if it contains no subpath of the form  $e\alpha_e(g)\overline{e}$ (for $g \in G_e$). A loop is \emph{cyclically reduced} if, in addition to being reduced, $e_n(g_ng_0)e_1$ is not of the form  $e\alpha_e(g)\overline{e}$.
\end{defn}

Every path is equivalent (by the relations for $F(\mathcal{G})$) to a reduced path, and similarly every loop is equivalent to both a reduced loop and a cyclically reduced loop. In general these reduced representations are not unique, although all equivalent (cyclically) reduced paths (or loops) will have the same edge structure. Note that a cyclically reduced loop might not be at the same vertex as the original loop.

\begin{defn}The \emph{fundamental group of $\mathcal{G}$} (at a vertex $v$) is the set of loops in $F(\mathcal{G})$ at $v$, and is denoted $\pi_1(\mathcal{G},v)$.
\end{defn}

The isomorphism class of this group does not depend on the vertex chosen. (In fact, the two groups obtained by choosing different base vertices are conjugate in the groupoid.)

We take the corresponding definition of the Bass-Serre tree:

\begin{defn}[Bass-Serre Tree]
Let $T$ be the graph formed as follows: the vertex set consists of `cosets' $G_{w}p$, where $p$ is a path in $F(\mathcal{G})$ from $w$ to $v$. There is an edge(-pair) joining two vertices $G_{w_1}p_1$ and $G_{w_2}p_2$ if $p_1=eg_{w_2}p_2$ or $p_2=eg_{w_1}p_1$ (where $g_w \in G_w$).
\end{defn}

The graph $T$ is a tree, usually called the Bass-Serre tree (or universal cover) for $\mathcal{G}$. Since loops at $v$ both start and finish at $v$, $\pi_1(\mathcal{G},v)$ acts on the right on the set of vertices, preserving adjacency.

\begin{defn}[Quotient graph of groups]
Given a group $G$ acting on a tree $T$, there is a \emph{quotient graph of groups} formed by taking the quotient graph from the action and assigning edge and vertex groups as the stabilisers of a representative of each orbit. Edge monomorphisms are then the inclusions, after conjugating appropriately if incompatible representatives were chosen.
\end{defn}

\begin{thm}[Structure theorem]
Up to isomorphism of the structures concerned, the processes of constructing the quotient graph of groups, and of constructing the fundamental group and Bass-Serre tree are mutually inverse.
\end{thm}

\subsection{Translation length}

The results in Section 4 require some calculations involving the translation length function for an action on a tree. This function was investigated in \cite{CullerMorgan1987Rtrees}; Section 1 of that paper proves many of its basic properties.

\begin{defn}[Translation length function]
	For a group $G$ acting on an ($\R$-)tree $T$ the \emph{translation length function} is $\Vert{-}\Vert:~G~\to~\R$ with $\Vert g \Vert = \inf_{x\in T} d(x,xg)$.
\end{defn}

If $g$ stabilises a point, then $\Vert g \Vert = 0$, and if $g$ is a hyperbolic element $\Vert g \Vert$ is the distance between a point on the axis and its image. Translation length is invariant under conjugation (that is, $\Vert h^{-1}gh \Vert =  \Vert g \Vert$). 
Also, if $T$ is a simplicial tree (with edge lengths equal to $1$), then the translation length function takes only integer values.

For the action of the fundamental group of a graph of groups on its Bass-Serre tree, using the definitions above, the translation length function is easy to calculate:

\begin{prop}
	\label{prop:translation_length_is_cyclic_reduced_length}
	Let $\mathcal{G}$ be a graph of groups, with fundamental group $G$, acting on its Bass-Serre tree $T$. For each element $g \in G$, the translation length $\Vert g \Vert$ is the path length of $g$ after cyclic reduction.
\end{prop}

\section{Sufficient conditions}

In this section we prove the following sufficient conditions for the automorphisms of a free product to have Property (FA):

\sufficient

Since these conditions require each factor to have Property (FA), they are certainly freely indecomposable and not infinite cyclic. So their automorphism group decomposes as $\Aut(G) = \FR(G) \rtimes \langle\Fact(G),\Perm(G)\rangle$ as described in Propostition \ref{prop:aut_G_pres}. First we will show that the quotient $\langle \Fact(G),\Perm(G) \rangle$ has Property (FA). In \cite{CornulierKar2011Wreathprods} Cornulier and Kar characterise the permutational wreath products with Property (FA). Their result is:
\begin{thm}[Cornulier and Kar, Theorem 1.1 of \cite{CornulierKar2011Wreathprods}]
\label{thm:wreath_prods}
Let $G$ be a group that is a permutational wreath product $G=A\wr_X B$ where $A \neq 1, X \neq \emptyset$ and $X$ has finitely many $B$-orbits each with more than one element. Then $G$ has Property (FA) if and only if $B$ has Property (FA) and $A$ has finite abelianisation and cannot be expressed as the union of a properly increasing sequence of subgroups.
\end{thm}

Since Proposition \ref{prop:structure_of_factor_perm_auts} gives us a decomposition of $\langle \Fact(G),\Perm(G)\rangle$ as a direct product of permutational wreath products, we may use this result to investigate this subgroup.

\begin{prop}
\label{prop:quotient_has_FA}
Letting $G$ be as in Theorem \ref{thm:Aut(G)_has_FA}, the subgroup generated by factor and permutation automorphisms has Property (FA).
\end{prop}
\begin{proof}
By Proposition \ref{prop:structure_of_factor_perm_auts} this is a direct product of permutation wreath products. If $G$ satisfies part (1) of Theorem \ref{thm:Aut(G)_has_FA} then each of them satisfies the hypotheses of Theorem \ref{thm:wreath_prods}: since each $n_i \geq 4$ the set $X$ is  non empty; $S_{n_i}$ acts transitively on it so there is only one orbit; and from the hypotheses of Theorem \ref{thm:Aut(G)_has_FA} these $\Aut(G_i)$ have finite abelianisation and cannot be expressed as the union of a properly increasing sequence of subgroups, and $S_{n_i}$ is finite so has Property (FA). So each wreath product has Property (FA). If $G$ satisfies part (2) of Theorem \ref{thm:Aut(G)_has_FA}, then automorphism group of the singleton factor has Property (FA) by assumption, and all others satisfy the hypotheses we need for \ref{thm:wreath_prods} just as above. So, in either case, we have a direct product of groups with Property (FA). Their direct product must also have Property (FA), by inductively applying the argument of Lemma \ref{lem:commuting_fixed_pt}.
\end{proof}




Next we show that, whenever $\Aut(G)$ acts on a tree, the subgroup $\FR(G)$ has a fixed point. Most of the arguments are similar, and proceed by finding `enough commutation' that various elliptic subgroups are forced to have common fixed points, but we write them out in full.

\begin{prop}
\label{prop:FR(G)_has_fix}
Let $G$ be as in part (1) of Theorem \ref{thm:Aut(G)_has_FA}.
Then any action of $\FR(G)$ on a tree which extends to $\FR(G) \rtimes \Perm(G)$ has a global fixed point.
\end{prop}

\begin{proof}
The subgroup $\FR(G)$ is generated by finitely many subgroups $(A,B)$ consisting of all partial conjugations $(A,b)$, where $A$ is fixed, but $b$ ranges over all of some other factor $B$. This is isomorphic (in fact, anti-isomorphic) to $B$, and so since $B$ has Property (FA), all such subgroups are elliptic. By Lemma \ref{lem:serre_intersection}, if their fixed point subtrees intersect pairwise then their intersection is non-empty. 

So we check all the possible pairs $(A,B)$ and $(C,D)$: (Different letters always represent different subgroups; some combinations cannot occur due to the fact the inner factor automorphisms are excluded.) \begin{enumerate}
\item $(A,B)$ and $(C,D)$: 
These commute (by Relation (\ref{rel:commuters})), and so since they are elliptic there must be a common fixed point by Lemma \ref{lem:commuting_fixed_pt}.
\item $(A,B)$ and $(C,B)$: 
These subgroups commute by Relation (\ref{rel:commuters}). Since both are elliptic, they must have a common fixed point by Lemma \ref{lem:commuting_fixed_pt}.
\item $(A,B)$ and $(A,D)$: 
Since there are (at least) four isomorphic copies of each factor group, there is some $C'$ (different to $A,B,D$) such that $C' \cong A$. Letting $\tau$ be the permutation interchanging $A$ and $C'$, then $(A,B),(A,D)$ and $\tau$ satisfy the conditions of Lemma \ref{lem:fixed_pt_new}(2): $(A,B)^\tau =(C',B)$ and $(A,D)^\tau = (C',D)$ both commute with both $(A,B)$ and $(A,D)$ (by Relation (\ref{rel:commuters})) and so have common fixed points by Lemma \ref{lem:commuting_fixed_pt}. So $(A,B)$ and $(A,D)$ have a common fixed point.
\item $(A,B)$ and $(B,D)$ (and, by symmetry $(A,B)$ and $(C,A)$): 
This time take $C' \cong B$ so that $A,B,C',D$ are all different. Let $\tau$ swap $C'$ and $B$, so conjugating by $\tau$ gives $(A,C')$ and $(C',D)$. Now $(C',D)$ commutes with both the original elements, and $(A,C')$ commutes with $(C,D)$ (all by relation (\ref{rel:commuters})), and so there are fixed points in common by Lemma \ref{lem:commuting_fixed_pt}. Also, $(A,B)$ and $(A,C')$ fit the hypotheses of Case (3), and so they have common fixed point. So $(A,B),(B,D)$ and $\tau$ satisfy Lemma \ref{lem:fixed_pt_new} and there is a common fixed point.
\item $(A,B)$ and $(B,A)$: 
Take $C' \cong B$ and $D' \cong A$, so that $A,B,C',D'$ are different factors. Then let $\tau$ swap $C'$ with $B$, and $D'$ with $A$. The images after conjugating by $\tau$ (which are $(D',C')$ and $(C',D')$ respectively) commute (by Relation (\ref{rel:commuters})) and so have common fixed points (by Lemma \ref{lem:commuting_fixed_pt}) with $(A,B)$ and $(B,A)$, and so $(A,B)$, $(B,A)$ and $\tau$ satisfy Lemma \ref{lem:fixed_pt_new} so these subgroups have a fixed point.
\end{enumerate}

These pairwise intersections satisfy Lemma \ref{lem:serre_intersection}, so have a non-empty intersection. This is fixed by every element of $\FR(G)$, and so since these subgroups generate, this intersection is fixed by $\FR(G)$ which must itself be elliptic.
\end{proof}


Before proving the second case, we cover one aspect of the proof in a lemma.

By analogy with a wreath product, we make the following definition.

\begin{defn}
Let $H$ and $K$ be groups, and equip $K$ with an action on a set $X$. The \emph{wreathed free product} of $H$ and $K$ (with respect to the given action) is the semidirect product $H^{*|X|}\rtimes K$, where the action of $K$ on $H^{*|X|}$ is to permute the free factors according to the action on $X$.
\end{defn}

The symmetries induced by the $K$-action have the effect of restricting the trees such groups can act on, as we see in the following lemma (restricting to the action of $S_n$ on a set of $n$ elements):

\begin{lem}
\label{lem:stupid_lemma}
Suppose the wreathed free product $H^{*n}\rtimes S_n$ acts on a tree $T$, such that the free factor $H_1$ fixes a subtree $T_1$. Then\begin{enumerate}[label=(\arabic*)]
\item Each factor $H_i$ fixes a subtree $T_i$, and these are permuted by the action of $S_n$ on $T$.
\item There are branch points (vertices) $v_i \in T_i$ such that \begin{itemize}
\item $d(v_i,v_j)=d(T_i,T_j)$ for all $i,j$
\item The $v_i$ are permuted by the action of $S_n$ on $T$
\item The $v_i$ have a common branch point $w$, which is the midpoint of each geodesic $(v_i,v_j)$ (for $i\neq j$)
\end{itemize}
Let $T'$ be the convex hull of the $v_i$, shown in Figure \ref{fig:Tprime}.
\item If in addition the tree is a $\Z$-tree, or the group $H$ is finitely generated, there are elements $h_i\in H_i$ such that $\Fix(h_i)\cap T'=v_i$. These can be chosen so they are permuted by the action of $S_n$ on $H^{*n}$.
\item If $T'$ is not a single point, and (3) occurs, then (viewing $h_1$ and $h_2$ as their images in $\Isom(T)$) both $h_1h_2$ and $h_2h_1$ are hyperbolic, and there is no isometry $f \in \Isom(T)$ such that $f^{-1}h_1h_2f=h_2h_1$, $wf=w$ and $v_1f=v_1$.   
\end{enumerate}
\end{lem}

\begin{figure}[ht]
\center
\begin{tikzpicture}
\draw[thick] (0,0) -- (90:2);
\draw[thick] (0,0) -- (60:2);
\draw[thick] (0,0) -- (30:2);
\draw[thick] (0,0) -- (120:2);
\draw[thick] (0,0) -- (330:2);

\node at (90:2.3) {$v_1$};
\node at (60:2.3) {$v_2$};
\node at (30:2.3) {$v_3$};
\node at (330:2.5) {$v_i$};
\node at (120:2.3) {$v_n$};
\node[below left] at (0,0) {$w$};

\draw[thick,dotted] (315:1.5) arc (315:130:1.5);
\draw[thick,dotted] (20:1.5) arc (380:340:1.5);
\end{tikzpicture}
\caption{The graph $T'$ described in Lemma \ref{lem:stupid_lemma}(2)}
\label{fig:Tprime}
\end{figure}
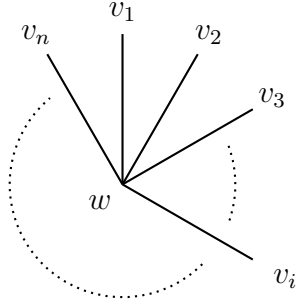

\begin{proof}
Write the elements of $S_n$ as (eg) $\sigma_{(12)}$, denoting that it swaps $H_1$ and $H_2$.
\begin{enumerate}[label=(\arabic*)]
\item Since $\sigma_{(1i)}^{-1}H_{1}\sigma_{(1i)}=H_i$, we must have that $H_i$ fixes precisely $T_i:=T_{1}\sigma_{(1i)}$, which in particular is non empty.
\item If $n=1$ then we may choose any global fixed point for the $H$-action. If any (and therefore every) pair has a common fixed point, in fact the $H^{*n}$ action must be elliptic. Since the $S_n$ action permutes the $T_i$, it acts on their (non-empty) intersection. Since it is finite, it does so with a fixed point, which will be a fixed point for the whole action. Let every $v_i$ be this fixed point; then this one point subtree works.

Otherwise, since $S_n$ acts $2$-transitively on the set $\{T_i\}$, and is acting by isometries, we have that $d(T_i,T_j)=\lambda(1-\delta_{ij})$ where $\lambda$ is a positive constant and $\delta_{ij}$ is the Kronecker delta. Let $v_{ij}$ (with $i\neq j$) be the nearest point in $T_i$ to $T_j$. In fact this is the same point as $j$ varies, since if there were $j,k$ such that $v_{ij}$ and $v_{ik}$ were different, we would have $d(T_j,T_k) = d(v_{ji},v_{ij}) + d(v_{ij},v_{ik})+d(v_{ik},v_{ki}) > d(T_i,T_j)+d(T_i,T_k)$. Since all three distances are equal, this isn't the case. Call this common nearest point $v_i$; then we have that $d(v_i,v_j)=d(T_i,T_j)$ as we wanted. Since the action is by isometries, $d(T_1,T_j)=d(v_1,T_j)=d(v_1\sigma_{(1i)},T_j\sigma_{(1i)})=d(v_1\sigma_{(1i)},T_{j(1i)})=d(T_i,T_{j(1i)})$. So $v_1\sigma_{(1i)}$ is $v_i$, and they are permuted by the $S_n$ action.

If $n=2$ then since $\sigma_{(12)}$ must swap $T_1$ and $T_2$, it will invert the geodesic. So (after subdividing if necessary) there must be a fixed vertex $w$ at the same distance from $v_1$ and $v_2$. Otherwise, it is an induction argument with base case $n=3$.

For $n=3$ consider the vertices $v_1,v_2,v_3$, and their y-point, $w$. This tripod contains three geodesics (joining the $v_i$) which must all be the same length, so each arm of the tripod (ie each distance $d(w,v_i)$) is the same length. For the induction step, consider the tripod for $v_1,v_2$ and $v_n$ together with the star for $v_1,\dots,v_{n-1}$. Both subtrees contain the geodesic $(v_1,v_2)$, and so in particular contain $w$. In fact, $w$ must be the y-point for the tripod, since it is still the midpoint of $(v_1,v_2)$. So $d(w,v_n)$ is again the same length, and $w$ is the y point for any triple of the $v_i$.



\item For each $h \in H_1$ consider the intersection $\Fix(h)\cap T'$. This is a collection of closed sets whose intersection must be the point $v_1$. In fact, we can consider a generating set for $H_i$ since the intersection of the generators' fixed subtrees is precisely $\Fix(H_1)$. So if $H_i$ is finitely generated, this can be a finite collection, and so one (or more) of the sets must only contain $v_1$. Alternatively, if the tree is a $\Z$-tree, then the diameter of each set must be an integer. But then in order for the intersection to consist only of the boundary point $v_1$, at least one of the sets must also be just $v_1$. Choose an $h_1 \in H_1$ such that $Fix(h_1)\cap T'=v_1$, and let $h_i \in H_i$ be $\sigma_{(1i)}^{-1}h_1\sigma_{(1i)}$ for each $i$. By definition, the $h_i$ are permuted by the action of $S_n$, and $Fix(h_i)\cap T' = Fix(h_1)\sigma_{(1i)} \cap T' =(Fix(h_1)\cap T')\sigma_{(1i)}=v_1\sigma_{(1i)}=v_i$ as needed.

\item If $T'$ is not a single point, then $h_1$ and $h_2$ have no common fixed point, and so the element $h_1h_2$ is hyperbolic: its axis includes the geodesic $(v_1,v_2)$, and its translation length is $2d(v_1,v_2)$. The same is true of the element $h_2h_1$, although it translates the other way along the common segment. 
Suppose we have an isometry $f$ which fixes the vertices $v_1$ and $w$, and thus the segment $(v_1,w)$.
Consider how our two elements move this fixed segment: $h_2h_1$ must move it past $v_1$. But $f^{-1}h_1h_2f$ can't move it past $v_1$: $f^{-1}h_1h_2$ moves it along $\Axis(h_1h_2)$, and then the nearest point of $\Fix(f)$ is closer than (in fact, along the geodesic to) $v_1$ so the segment must stay the same side. But then $f^{-1}h_1h_2f \neq h_2h_1$, and so there is no isometry with the properties in the statement.
\qedhere
\end{enumerate}
\end{proof}

\begin{prop}
\label{prop:ugly_FR(G)_has_fix}
Let $G$ be as in part (2) of Theorem \ref{thm:Aut(G)_has_FA}. Then
any action of $\FR(G)$ on a tree which extends to $\FR(G) \rtimes \Perm(G)$ has a global fixed point.
\end{prop}

\begin{proof}
This is the same idea as for Proposition \ref{prop:FR(G)_has_fix}, but depends even more on having access to symmetries required by the permutation automorphisms. 

In any action of $\FR(G)$ on a tree the following three subgroups have global fixed points: \begin{enumerate}[label=(\roman*)]
\item The subgroup generated by all partial conjugations of repeated factors by repeated factors - this is what was proved in Proposition \ref{prop:FR(G)_has_fix}
\item The subgroup generated by all partial conjugations where the conjugating group is the non-repeated factor. This is a direct product of several copies of that factor, and so will have Property (FA) since it does.
\item Each subgroup generated by $(K,h_i)$, where $K$ is the non-repeating factor, and the $H_i$ are all $n \geq 4$ copies of some repeating factor. This generates a group isomorphic to $H^{*n}$, and we use Lemma \ref{lem:stupid_lemma}. The permutation automorphisms normalise this subgroup, and so together with it generate a wreathed free product $H^{*n}\rtimes S_n$, which must act on any tree the full automorphism group does. Suppose it is not elliptic, so part (3) holds since we are considering actions on a $\Z$-tree. Let $(K,h_1)$ and $(K,h_2)$ be two of the elements described in that part. By the final commutation relation for $\FR(G)$, we have that $[(K,h_1)(H_2,h_1),(K,h_2)]=1$. Expanding (and moving some commuting elements past each other) this gives that $(K,h_1)(K,h_2)=(H_2,h_1)^{-1}(K,h_2)(K,h_1)(H_2,h_1)$. In addition, $(H_2,h_1)$ commutes with and so has common fixed points with all $(K,h_i)$ with $i \neq 2$, and so fixes $v_1$ and the central vertex. But then the isometry induced by $(H_2,h_1)$ has precisely the properties forbidden by part (4), so we have a contradiction. So in fact this subgroup must be elliptic.
\end{enumerate}

All subgroups of the form $(A,B)$ are contained in one of these subgroups, so must themselves be elliptic. As before, we check that any pair of these subgroups have a common fixed point. Pairs drawn from the same subgroup are already done, so we check the cases where they are drawn from different subgroups. Some cases are by commuting subgroups, others rely on Lemma \ref{lem:fixed_pt_new} and so are similar to the technique used in the previous result, and others need the use of the final relation. Denote by $K$ the factor occuring once, and by $A,B, \dots$ any of the factors that appear at least four times. As before, different letters denote different factors.\begin{enumerate}
\item $(A,B)$ and $(C,K)$: These commute and therefore have a common fixed point (by Relation \ref{rel:commuters} and Lemma \ref{lem:commuting_fixed_pt}).
\item $(A,B)$ and $(A,K)$: Let $\tau$ be the permutation swapping $A$ and some $C' \cong A$ (different to $A$ and $B$). Conjugating by $\tau$ gives $(C',B)$ and $(C',K)$, which both commute with both original subgroups. So by Lemma \ref{lem:fixed_pt_new} we get that our elements have a common fixed point.
\item $(A,B)$ and $(B,K)$ In the second case, let $C' \cong B$, different to $A$ and $B$, and let $\tau$ swap $B$ and $C'$. After conjugating, both have common fixed points with the original subgroups: $(A,C')$ and $(A,B)$ by case (3) of Proposition \ref{prop:FR(G)_has_fix} and the rest since they commute. So we satisfy Lemma \ref{lem:fixed_pt_new} and there is a common fixed point.
\item $(A,B)$ and $(K,C)$ or $(A,B)$ and $(K,B)$ commute so there will be a common fixed point. 
\item $(A,B)$ and $(K,A)$: In this case we need $\tau$ to swap $A$ and a $C' \cong A$, giving $(C',B)$ and $(K,C')$. These both have common fixed points with both original elements, in three cases because they commute, and in the fourth because of Case (3) above. So we have the common fixed points we need to once again deploy Lemma \ref{lem:fixed_pt_new} to give us a common fixed point.
\item $(K,A)$ and $(B,K)$: Consider $[(B,a)(K,a),(B,k)]=1$ (another of the second kind of commutator relation). Since they commute (and are elliptic), $(B,a)(K,a)$ is elliptic. 
Also, $(B,k)$ is elliptic, and so since these elements commute (and are elliptic) there is a common fixed point between $(B,a)(K,a)$ and $(B,k)$. But this means there is a common fixed point between all three elements, and so in particular between $(K,a)$ and $(B,k)$. We now apply Lemma \ref{lem:naomi_intersection} twice: first, fix $a \in A$ and vary $k \in K$ to see that $\Fix(K,a)$ and $\Fix(B,K)$ have non empty intersection for all $a \in A$. Then this gives that $\Fix(K,A)$ and $\Fix(B,K)$ have non-empty intersection, as we wanted.
\item $(A,K)$ and $(K,A)$: Let $B' \cong A$ (but different), and let $\tau$ be the permutation automorphism swaping $A$ and $B'$. Again, these satisfy Lemma \ref{lem:fixed_pt_new}: the four pairs are $(A,K)$ and $(B',K)$ which commute; $(K,A)$ and $(K,B')$ which are both in the third kind elliptic subgroup identified at the start of the proof; $(A,K)$ and $(K,B')$, and $(K,A)$ and $(B',K)$ which satisfy the previous case. So this final pair also have a common fixed point.
\item $(K,A)$ and $(K,B)$ where $A \ncong B$. Consider $[(K,a)(B,a),(K,b)]=1$: just as above, this gives a common fixed point between $(K,a)$ and $(K,b)$ and then applying Lemma \ref{lem:naomi_intersection} gives common fixed points between $\Fix(K,A)$ and $\Fix(K,B)$. \qedhere
\end{enumerate}

%
%
\end{proof}

Propositions \ref{prop:quotient_has_FA}, \ref{prop:FR(G)_has_fix} and \ref{prop:ugly_FR(G)_has_fix} provide the proof of Theorem \ref{thm:Aut(G)_has_FA}, as follows:

\begin{proof}[Proof of Theorem \ref{thm:Aut(G)_has_FA}]
An action of $\Aut(G)$ on a tree defines an action of $\FR(G)$ on the same tree. Since $\FR(G) \rtimes \Perm(G) \leq \Aut(G)$, this action must extend to the permutation automorphisms. So by Proposition \ref{prop:FR(G)_has_fix} or \ref{prop:ugly_FR(G)_has_fix} this subgroup is elliptic. Now consider $v \in \Fix(\FR(G))$: we have that $vhg=vg'h=vh$ for all $g\in \FR(G),h \in \langle\Fact(G), Perm(G) \rangle$, where $g'=hgh^{-1}\in \FR(G)$. So $\langle\Fact(G), \Perm(G)\rangle$ acts on the fixed point set of $\FR(G)$. Since it has Property (FA) by Proposition \ref{prop:quotient_has_FA} that action will have a fixed point, which must be a fixed point for the whole action. So $\Aut(G)$ also has Property (FA).
\end{proof}


\section{Necessary conditions}
The results in this section, taken together, will prove all parts of Theorem \ref{thm:has_not_FA}. First we deal with the (shorter) parts (2) and (3), and then afterwards part (1).

\necessary

\begin{prop}
\label{prop:no_FA_for_one_factor}
Let $G$ be a free product of (freely indecomposable) groups, with no infinite cyclic factors. Suppose there is some free factor $H$ whose isomorphism class appears exactly once in the Grushko decomposition, and $\Aut(H)$ does not have Property (FA). Then $\Aut(K)$ does not have Property (FA).
\end{prop}
\begin{proof}
By \ref{prop:aut_G_pres} (and since there are no infinite cyclic factors) the group $\langle\Fact(G), \Perm(G)\rangle$ is a quotient of $\Aut(G)$. Then by part (3) of Proposition \ref{prop:structure_of_factor_perm_auts}, one of the direct summands of this group is $\Aut(H)$. So $\Aut(H)$ is a quotient of $\Aut(G)$. Since $\Aut(H)$ has an action on a tree without global fixed point, the same is true of $\Aut(G)$.
\end{proof}

\begin{prop}
\label{prop:no_FA_for_geq_two_factors}
Let $G$ be a free product of (freely indecomposable) groups, with no infinite cyclic factors. Suppose there is a free factor $H$ that (up to isomorphism) appears at least two times in the decomposition, and $\Aut(H)$ does not have finite abelianisation or can be expressed as a union of a properly increasing sequence of subgroups. Then $\Aut(G)$ does not have Property (FA).
\end{prop}
\begin{proof}
Again, by \ref{prop:aut_G_pres} and since there are no infinite cyclic factors $\langle\Fact(G), \Perm(G)\rangle$ is a quotient of $\Aut(G)$. Then again using part (3) of Proposition \ref{prop:structure_of_factor_perm_auts}, one of the direct summands of this group is $\Aut(H)^n \wr S_n$. By Theorem \ref{thm:wreath_prods} this does not have Property (FA). So since $\Aut(G)$ has a quotient that does not have Property (FA), neither does $\Aut(G)$.
\end{proof}

For part (1) of Theorem \ref{thm:has_not_FA}, we will extend the action of $G$ on a Bass-Serre tree to its (outer) automorphisms. It is useful to view an action of a group $G$ by isometries on a tree $T$ as a homomorphism $G \to \Isom(T)$.

We recall a theorem of Culler and Morgan on $\R$-trees and translation length:

\begin{thm}[Culler-Morgan, 3.7 of \cite{CullerMorgan1987Rtrees}]
\label{thm:Culler_Morgan}
If a group $G$ acts minimally and without preserving the orientation of an invariant line on $\mathbb{R}$-trees $T_1$ and $T_2$ with the same translation length function, then there is a unique $G$-equivariant isometry $f:T_1 \to T_2$. That is, $f$ is the unique isometry such that with $f^*:~Isom(T_1)~\to~Isom(T_2)$ defined by $\varphi f^* =f^{-1}\varphi f$ this diagram commutes.

\centerline{\xymatrix@R-20pt{
& Isom(T_1) \ar[dd]^{f^*} \\
G \ar[ur] \ar[dr] & \\
& Isom(T_2) \\
}}
\end{thm}

Their condition on the action is that it is minimal and semisimple (either irreducible, a single point, dihedral or a shift), and for uniqueness that it is not a shift. They also do not give the interpretation as a commutative diagram.

We will consider the following subgroup of automorphisms:

\begin{defn}
	Suppose $G$ is a group acting on a tree $T$. Let $\Aut_T(G)$ be the subgroup of $\Aut(G)$ that preserves the translation length function of the action.
\end{defn}

The uniqueness part of their theorem allows us to prove the following corollary:

\begin{cor}
\label{cor:trans_length_pres}
Given a group $G$ acting minimally and without preserving the orientation of an invariant line on a $\R$-tree $T$, with associated translation length function, then $\Aut_T(G)$ acts by isometries on $T$. Further, \begin{enumerate}
\item The inner automorphism given by conjugating by $g$ induces the same isometry as $g$. So if the original action has no fixed points, the same is true for this action.
\item The action of $\Aut_T(G)$ is compatible with the action of $G$, in the sense that the subgroup $G \rtimes \Aut_T(G)$ of the holomorph acts on T with the given actions of each factor.
\item If the original tree was a $\Z$-tree then the action constructed is also an action on a $\Z$-tree, after subdividing if necessary to remove edge inversions.
\end{enumerate}
\end{cor}



Groups acting on trees may have a non-trivial centre, for example in $\SL(2,\Z)$ the centre has order $2$. However, if the action has at least two axes, or a single axis and an elliptic element that does not preserve its orientation, then the centre of the group must be in the kernel of the action. So since two elements inducing the same inner automorphism must already have the same image in $\Isom(T)$, the isometry described in (1) is unique.


\begin{proof}[Proof of Corollary \ref{cor:trans_length_pres}]
Given an action~$\cdot: G \to \Isom(T)$, and any automorphism $\varphi$ of $G$, there is another action defined by $*_\varphi = \varphi \circ \cdot: G \to G \to \Isom(T)$. That is, $t *_\varphi g = t \cdot (g \varphi)$.


For any $\varphi \in \Aut_T(G)$, this action will have the same translation length function as $\cdot$. So we may apply Theorem \ref{thm:Culler_Morgan} to the actions $\cdot$ and $*_\varphi$ to give a unique isometry $f_\varphi$ of $T$ (corresponding to the automorphism $\varphi$) such that the following diagram commutes:

\centerline{\xymatrix{
G \ar[r]^{\cdot} \ar[d]_\varphi \ar@{-->}[dr]^{*_\varphi} & Isom(T) \ar[d]^{f^*_\varphi} \\
G \ar[r]^\cdot & Isom(T)
}}

These isometries do give an action: for $*_1 = \cdot$, the identity map on $T$ is an equivariant isometry of $T$ making the diagram commute. So by uniqueness, $f_1=Id$. Now for $\varphi,\psi \in \Aut_T(G)$, consider this diagram:

\centerline{\xymatrix{
G \ar[r]^{\cdot} \ar[d]_\varphi \ar@{-->}[dr]^{*_\varphi} & Isom(T) \ar[d]^{f^*_\varphi} \\
G \ar[r]^{\cdot} \ar[d]_\psi \ar@{-->}[dr]^{*_\psi} & Isom(T) \ar[d]^{f^*_\psi} \\
G \ar[r]^\cdot & Isom(T)
}}

Here, the top square shows the equivariant isometry induced by $\varphi$, and the bottom square that by $\psi$. We want to consider the element $\varphi\psi$, which should also induce a unique equivariant isometry. However, from the diagram, the composition $f_\varphi f_\psi$ is just such an equivariant isometry, and so it must be the unique $f_{\varphi\psi}$.

To see (1): let $g$ be some element of $G$, and consider the inner automorphism that conjugates by $g$ (called $\delta(g)$). 
In this case, the usual action of $g$ is an equivariant isometry for the conjugation, since we have that $(x\cdot g) *_{\delta(g)}h = (x \cdot g) \cdot (g^{-1}hg) = x \cdot hg = (x \cdot h) \cdot g$. As a commutative diagram (where the right hand arrow is induced by the action of $g$):

\centerline{\xymatrix{
G \ar[r]^{\cdot} \ar[d]_{\delta(g)} \ar@{-->}[dr]^{*_{\delta(g)}} & Isom(T) \ar[d] \\
G \ar[r]^\cdot & Isom(T)
}}

So, by uniqueness, $f_{\delta(g)}=\cdot g$. 

To see (2): we need to check that the isometries corresponding to $(g\varphi)$ (as an element of $G$) and $\varphi^{-1}g\varphi$ are the same for all $g \in G$ and $\varphi \in \Aut_T(G)$. This is immediate from the commutative diagram in the statement of Theorem \ref{thm:Culler_Morgan} with the actions $\cdot$ and $\ast_\varphi$, since the downwards arrow is then precisely conjugation by the isometry corresponding to $\varphi$.

To see (3): The induced isometries must send branch points to branch points (of the same valence). In a $\Z$-tree, since branch points are vertices and all other vertices are at integer distance, the vertex set must be preserved by the induced isometries. In the case where $T$ is a single line, there must be a vertex stabilised by an orientation reversing element. The induced isometries must send this to another such point, and in a $\Z$-tree these are all vertices. Again, all other vertices are at integer distance, and so the vertex set is preserved. So the action of $\Aut_T(G)$ is still by an action on the $\Z$-tree, as we needed.
\qedhere

\end{proof}

We use this corollary to prove the first part of Theorem \ref{thm:has_not_FA}. First, we prove special cases where there are \emph{only} two or three free factors (satisfying the conditions of the theorem) and then use the discussion of characteristic subgroups to extend these results. In the two factor case we construct an action of the full automorphism group on a Bass-Serre tree for the group; in the three factor case it is an action of the outer automorphism group.

\begin{prop}
\label{prop:two_factors}
If $G = H * K$ then $Aut(G)$ does not have Property FA. (Here one or both factor groups may be $\Z$ without affecting the result.)
\end{prop}
\begin{proof}
If neither $H$ nor $K$ are infinite cyclic, realise $G$ as the fundamental group of the graph of groups shown in Figure \ref{fig:freeprod_HK}. Consider the action of $G$ on the Bass-Serre tree for this graph of groups. We want to check that the translation length is preserved by every automorphism, which requires us first to calculate it. An elliptic element is a conjugate of an element of a factor group. None of the generating automorphisms change this, and so all automorphic images of elliptic elements are themselves elliptic.

\begin{figure}[h]
\begin{subfigure}{.49\textwidth}
\center
\begin{tikzpicture}[scale=1.5]
\draw[thick,middlearrow={.5}{>}{$e$}{below}] (0,0) -- (2,0); 
\draw[fill,color=red] (0,0) circle [radius=0.1]; \node at (0,-0.4) {$H$};
\draw[fill,color=blue] (2,0) circle [radius=0.1]; \node at (2,-0.4) {$K$};
\end{tikzpicture}
\caption{Graph of Groups for $H*K$}
\label{fig:freeprod_HK}
\end{subfigure}
\begin{subfigure}{.49\textwidth}
\center
\begin{tikzpicture}[scale=1.5]
\draw[thick,middlearrow={.55}{>}{$e$}{left,rotate=75}] (0,0) arc (0:350:4mm);
\draw[fill,color=red] (0,0) circle [radius=0.1]; \node at (0.4,0) {$H$};
\end{tikzpicture}
\caption{Graph of Groups for $H*Z$}
\label{fig:freeprod_HZ}
\end{subfigure}
\caption{Graphs of groups realising each $G$ in Proposition \ref{prop:two_factors}.}
\end{figure}
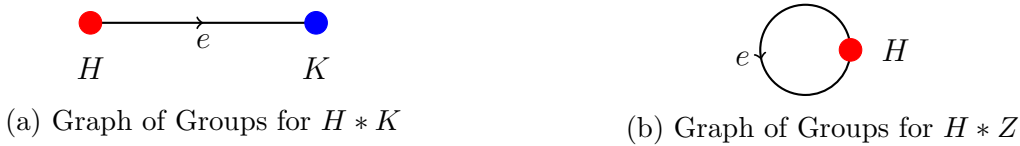
Any hyperbolic element can be cyclically reduced to a conjugate of the form $a_1b_1a_2b_2 \dots a_nb_n$, where $a_1$ and $b_n$ are non trivial. 
The path length of this cyclically reduced word is just its length: we use Proposition \ref{prop:translation_length_is_cyclic_reduced_length}, and note that since the original graph of groups had only, the path length of this conjugate doesn't change depending on which vertex was picked as the base point (and therefore which vertex group needs an edge in the path to get to it). In fact, this path length is just the length of the word. Generating automorphisms don't change the word length, after conjugating to get back to our required cyclically reduced form. This is obvious for factor automorphisms. There is a permutation automorphism ($\tau$) if and only if $H$ and $K$ are isomorphic: applying this won't change the length, although we do need to conjugate (by $a_1\tau$) to get back to our preferred form. This doesn't change the length, so translation length is unchanged. For the partial conjugation $(A,b)$ we have $a_1b_1a_2b_2 \dots a_nb_n \mapsto (b^{-1}a_1b)b_1(b^{-1}a_2b)b_2 \dots (b^{-1}a_nb)b_n \sim a_1b_1'a_2b_2' \dots a_nb_n'$ where each $b_i'=bb_ib^{-1}$. This is a cyclically reduced conjugate of the correct form and with the same length, and so the translation length is unchanged.

Since all the generators preserve the translation length function, by Corollary \ref{cor:trans_length_pres} there is an action of the automorphism group on the Bass-Serre tree that is without global fixed points. If $H$ and $K$ were isomorphic, so there was a permutation automorphism, that isometry will invert the edge in the fundamental domain, we need to pass to the barycentric subdivision; otherwise no subdivisions are necessary.

If one factor is $\Z$ (so $G= H * \Z$), we can use the same technique. A generating set for this automorphism group consists of the partial conjugations $H \mapsto H^x$ and $x \mapsto h^{-1}xh$ for all $h \in H$, the transvections $x \mapsto hx$ for all $h \in H$, and the factor automorphisms $\Aut(H)$ and $x \mapsto x^{-1}$. Realise $G$ as the fundamental group of the graph of groups in Figure \ref{fig:freeprod_HZ} and consider the action of $G$ on its Bass-Serre tree. Elliptic elements are in some conjugate of $H$, and are sent to some other conjugate of $H$ by all of the generators. Hyperbolic elements have a cyclically reduced conjugate of the form $h_1x^{n_1} \dots h_kx^{n_k}$. By Proposition \ref{prop:translation_length_is_cyclic_reduced_length} the translation length of this element is $\sum |n_i|$. Factor automorphisms of $H$ don't affect this; conjugating $H$ by $x$ is an inner automorphism so can't change the translation length. Replacing $x$ with any of its images, after conjugating by $h^{-1}$ if necessary to return to a cyclically reduced conjugate of the preferred form, has the same absolute exponent sum, and so the translation length is unchanged. So every element of the automorphism group is length preserving, and so it too acts on the Bass-Serre tree by Corollary \ref{cor:trans_length_pres}. (In this case no inversions are introduced, so we do not need to subdivide the tree.)

If both factor groups are infinite cyclic, then $\Aut(G)$ is just $\Aut(F_2)$, which does not have Property FA.
\end{proof}

For the case with three isomorphic factors, we will find an action of the outer automorphism group on a tree, similar to that given in 4.1 and 4.2 of \cite{CollinsGilbert1990AutFreeProd} for three non-isomorphic factors, and in \cite{Leder2018FA} for finite cyclic groups. 

\begin{restatable}{prop}{outerauts}
\label{prop:outerauts}
	If $G = A*B*C$ is a free product of three copies of some (freely indecomposable) group, then $\Out(G)$ has a presentation given by:
	\begin{description}
	\item[Generators:]$(A,b),(B,c),(C,a), \Aut(A)^3, \sigma_{(123)},\sigma_{(12)}$
	\item[and relations:]
	\setcounter{equation}{0} \begin{align}
	(A,b)(A,b')&=(A,b'b) \text{,  etc.} \\
	\varphi \varphi'&= \varphi'' \qquad \text{from the direct product structure} \\
	\sigma_{(123)}^3=1,\sigma_{(12)}^2&=1, (\sigma_{(123)}\sigma_{(12)})^2=1 \qquad \text{(relations for $S_3$)} \\
	\varphi^{-1} (A,b)\varphi&=(A, b\varphi) \text{,  etc.} \\
	\sigma_{(12)}^{-1} (A,b)\sigma_{(12)}&= \gamma(a^{-1})(C,a^{-1}) \\
	\sigma_{(12)}^{-1} (B,c)\sigma_{(12)}&= \gamma(c^{-1})(B,c^{-1}) \\
	\sigma_{(12)}^{-1} (C,a)\sigma_{(12)}&= \gamma(b^{-1})(A,b^{-1}) \\
	\sigma_{(123)}^{-1} (A,b) \sigma_{(123)} &= (B,c) \text{,  etc.} \\
	\sigma^{-1}\varphi\sigma&=\varphi' \qquad\text{from the wreath product structure}
	\end{align}
	\end{description}
	Here $\gamma(a)$ means the inner factor automorphism conjugating $A$ by $a$, $\varphi$ is reserved for factor automorphisms, and $\sigma$ for permutation automorphisms. The relations should be taken to range over all appropriate generators.
\end{restatable}

A proof of this presentation is given as an appendix, since it closely follows the proof in \cite{CollinsGilbert1990AutFreeProd} for three non-isomorphic groups.

This presentation gives a semidirect product decomposition of $\Out(G)$ as $(\hat{G} \rtimes \Fact(G)) \rtimes \Perm(G)$, where $\hat{G}$ is isomorphic to $G$ but generated by the $(A,b)$,$(B,c)$ and $(C,a)$ (denote these factor groups by $\hat{B}$,$\hat{C}$ and $\hat{A}$ respectively) and the actions are the actions from the original semidirect decomposition of $\Aut(G)$. However, whenever the factors are not abelian, the order of evaluation is now important, since $\Perm(G)$ does not normalise $\hat{G}$ in the presence of inner factor automorphisms.

\begin{prop}
\label{prop:three_factors}
If $G=A*B*C$ is a free product of three copies of some (freely indecomposable) group, then $\Out(G)$ (and therefore $\Aut(G)$) acts on a tree without global fixed points.
\end{prop}


\begin{figure}[ht]
\begin{subfigure}{.49\textwidth}
\center
\begin{tikzpicture}
\draw[thick,middlearrow={.55}{>}{$e_1$}{left,rotate=-90}] (0,0) -- (90:2);
\draw[thick,middlearrow={.55}{>}{$e_2$}{above left,rotate=150}] (0,0) -- (210:2);
\draw[thick,middlearrow={.55}{>}{$e_3$}{above right,rotate=30}] (0,0) -- (330:2);
\draw[fill,color=black] (0,0) circle [radius=0.1]; \node at (0,-0.4) {};
\draw[fill,color=red] (90:2) circle [radius=0.1]; \node at (90:2.4) {$\hat{A}$};
\draw[fill,color=red] (210:2) circle [radius=0.1]; \node at (210:2.4) {$\hat{B}$};
\draw[fill,color=red] (330:2) circle [radius=0.1]; \node at (330:2.4) {$\hat{C}$};
\end{tikzpicture}
\caption{$\hat{G} \leq \Out(G)$}
\label{fig:freeprod_ABC}
\end{subfigure}
\begin{subfigure}{0.49\textwidth}
\center
\begin{tikzpicture}
\draw[thick,middlearrow={.55}{>}{$\Fact(G)$}{left,rotate=-90}] (0,0) -- (90:2);
\draw[thick,middlearrow={.55}{>}{$\Fact(G)$}{above left,rotate=150}] (0,0) -- (210:2);
\draw[thick,middlearrow={.55}{>}{$\Fact(G)$}{above right,rotate=30}] (0,0) -- (330:2);
\draw[fill,color=black] (0,0) circle [radius=0.1]; \node at (0,-0.6) {$\Fact(G)$};
\draw[fill,color=red] (90:2) circle [radius=0.1]; \node at (90:2.4) {$\hat{A} \rtimes \Fact(G)$};
\draw[fill,color=red] (210:2) circle [radius=0.1]; \node at (210:2.4) {$\hat{B} \rtimes \Fact(G)$};
\draw[fill,color=red] (330:2) circle [radius=0.1]; \node at (330:2.4) {$\hat{C} \rtimes \Fact(G)$};
\end{tikzpicture}
\caption{$\hat{G} \rtimes \Fact(G)$}
\label{fig:prod_ABC_facts}
\end{subfigure}
\par\bigskip
\begin{subfigure}{\textwidth}
\center
\begin{tikzpicture}
\draw[thick,middlearrow={.5}{>}{$\Fact(G)\rtimes S_2$}{above}] (0,0) -- (4,0); 
\draw[fill,color=black] (0,0) circle [radius=0.1]; \node at (0,-0.4) {$\Fact(G)\rtimes S_3$};
\draw[fill,color=red] (4,0) circle [radius=0.1]; \node at (4,-0.4) {$(\hat{A}\rtimes \Fact(G))\rtimes S_2$};
\end{tikzpicture}
\caption{$(\hat{G}\rtimes \Fact(G)) \rtimes S_3$}
\label{fig:prod_ABC_facts_perms}
\end{subfigure}
\caption{The graphs of groups at each stage of Proposition \ref{prop:three_factors}}
\end{figure}
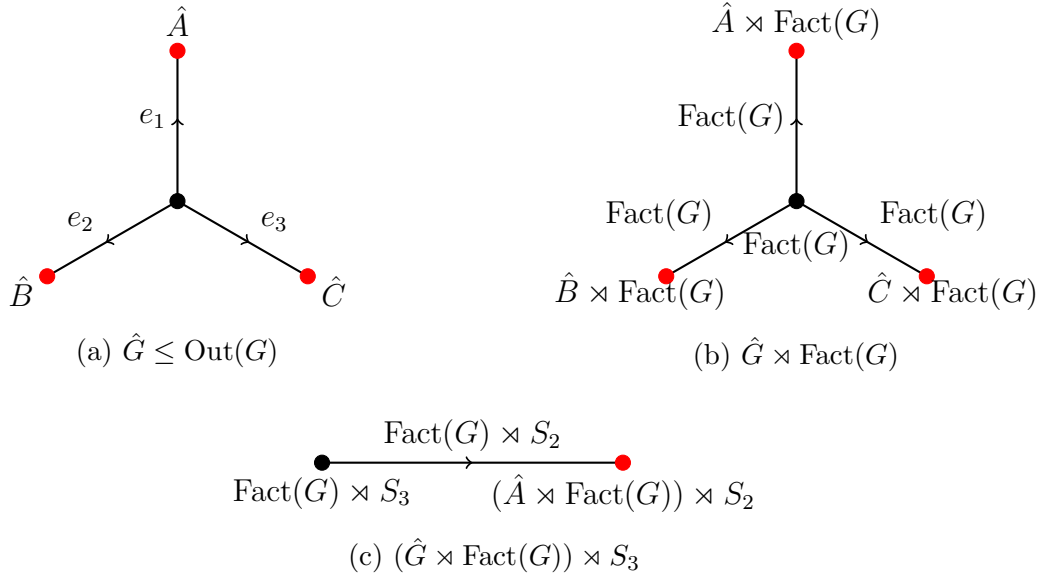


\begin{proof}
We will construct an action at each stage of the semidirect product decomposition.

Consider the tripod graph of groups for $\hat{G}$ (shown in Figure \ref{fig:freeprod_ABC}), taking the central vertex to be the base point. Call the Bass-Serre tree for this graph of groups $T$. Any elliptic word can be cyclically reduced to a single letter - a path of zero length (as expected). The translation length of any hyperbolic word is twice the length of a cyclically reduced conjugate, since every letter will require traversing two edges. The factor automorphisms act by sending $(A,b)$ to $(A,b\varphi )$ (for example), and so they don't change the (cyclically reduced) word length. So the factor automorphisms are translation length preserving, and have an action on $T$. By Corollary \ref{cor:trans_length_pres}(2), this action is compatible with the action of $\hat{G}$, and so we have an action of $\hat{G} \rtimes \Fact(G)$ on $T$.

A quotient graph of groups for this action, taking the same fundamental domain, is shown in Figure \ref{fig:prod_ABC_facts}. The factor automorphisms preserve the subgroups $\hat{A}$, $\hat{B}$, and $\hat{C}$ and so fix the fundamental domain: the equivariance of the induced isometries means fixed points are sent to fixed points, so an automorphism preserving a subgroup will induce an isometry preserving its fixed point set. Since the fixed point sets in a an action with trivial edge stabilisers are single vertices, this means that the induced isometry fixes the same vertex. The central vertex must then be fixed in order to preserve adjacency. Also, no orbits are collapsed by this action, so the fundamental domain does remain the same.

A element of $\hat{G} \rtimes \Fact(G)$ may be written uniquely as $\varphi w$, where $\varphi$ is an element of $\Fact(G)$ and $w$ is an element of $\hat{G}$. Each conjugacy class has a representative with $w$ cyclically reduced: first write $w=h^{-1}gh$, where $g$ is cyclically reduced, and the last letter of $g$ and the first letter of $h$ are drawn from different factor groups. (So there are no reductions or concatenations to do except possibly at $h^{-1}g$.) We can then conjugate the element as a whole by $h^{-1}$, giving $\varphi w \sim h \varphi h^{-1}ghh^{-1} = \varphi (h\varphi)h^{-1}g$. The word $(h\varphi)h^{-1}g$ (after cancelling and concatenating as necessary, depending on how many terminal letters of $h$ are fixed by $\varphi$) is cyclically reduced, since we chose $h$ to ensure that it (and so also $(h\varphi)$) have a first letter drawn from a different factor group to the last letter of $g$.

So it is enough to calculate the translation length of $\varphi w$ when $w$ is cyclically reduced.
Using the graph of groups in Figure \ref{fig:prod_ABC_facts}, since $\varphi$ can be picked up at the same vertex as the first non trivial group element, and $w$ is cyclically reduced, the length of the (cyclically reduced) path for $\varphi w$ is just the same as that for $w$. So $\Vert\varphi w \Vert = \Vert w \Vert$.

Now we need to describe the effect of a permutation automorphism on the translation length. We have seen that is enough to understand it in the case where $w$ is cyclically reduced, so we restrict to this case. In general, there are inner factor automorphisms introduced by the permutation, which we will need to move past the rest of the word to get back to our standard form. This can't change the length or structure (in terms of a sequence of factor groups from which the elements have come) of the word, since they either fix each letter or replace it with a different letter from the same factor group. So $(\varphi w)\sigma = \varphi' w'$, where $\varphi'$ is a (likely different) element of $\Fact(G)$, and $w'$ is the image of $w$ after applying $\sigma$ and moving any inner factor automorphisms past it. Provided $w$ was cyclically reduced, $w'$ is also cyclically reduced and has the same length. So by the argument above, $\Vert (\varphi w)\sigma \Vert = \Vert w' \Vert = \Vert w \Vert = \Vert \varphi w \Vert$, and so the translation length is preserved by the permutation automorphisms.

So the permutation automorphisms are a subgroup of $\Aut_T(\hat{G}\rtimes\Fact(G))$, and so we may further extend the action to the full outer automorphism group $(\hat{G} \rtimes \Fact(G)) \rtimes \Perm(G)$, again by applying Corollary \ref{cor:trans_length_pres}(2).
\end{proof}

A quotient graph of groups for this action (giving the splitting) is shown in Figure \ref{fig:prod_ABC_facts_perms}. The effect of the permutation automorphisms is to collapse the orbits of the three outer vertices to one, while preserving the orbit of the central vertex. So there are two orbits of vertices and one orbit of edges: the edge is stabilised by $\Fact(G) \rtimes C_2$, and the vertices by $(\hat{A} \rtimes \Fact(G))\rtimes C_2$ and by $\Fact(G) \rtimes S_3$.

We are now in a position to prove the final part of Theorem \ref{thm:has_not_FA}:

\begin{cor}
\label{cor:No_FA_for_counts}
Suppose $G$ is a free product of freely indecomposable groups, such that any of the following occur: \begin{itemize}
\item the free rank is exactly $2$;
\item the free rank is exactly $1$, and another free factor appears exactly once
\item $G$ has no infinite cyclic factors and either a free factor appears exactly two or three times, or any two free factors appear exactly once.
\end{itemize}
Then $Aut(G)$ does not have Property FA.
\end{cor}

\begin{proof}
The normal subgroup generated by all other free factors is characteristic, since they contain all representatives of their isomorphism class. So there is a homomorphism from $Aut(G)$ onto $Aut(H)$, where $H$ is the subgroup generated by the free factors described in the hypotheses. We have that $\Aut(H)$ acts on a tree by Proposition \ref{prop:two_factors} if $H$ has two free factors, or that the quotient $\Out(H)$ (and therefore the group $\Aut(H)$) does by Proposition \ref{prop:three_factors} if there are three. In either case, we have an action (without global fixed points) of a quotient of $Aut(G)$ on a tree, and so $\Aut(G)$ also acts on that tree without global fixed points. So $\Aut(G)$ does not have Property (FA).
\end{proof}

The proof of Theorem \ref{thm:has_not_FA} is just assembling the proofs in this section:
\begin{proof}[Proof of Theorem \ref{thm:has_not_FA}]\hfill
\begin{enumerate}
\item This is (3) of Corollary \ref{cor:No_FA_for_counts}.
\item This is Proposition \ref{prop:no_FA_for_one_factor}
\item This is Proposition \ref{prop:no_FA_for_geq_two_factors} \qedhere
\end{enumerate}
\end{proof}

However, all of these groups (assuming the free product is non-trivial) do have finite index subgroups that admit actions on trees:

\begin{prop}
\label{prop:no_prop_T}
Suppose $G$ is a (finite, non-trivial) free product, where each factor is freely indecomposable and not infinite cyclic. Then there is a finite index subgroup of $\Aut(G)$ that does not have Property (FA).
\end{prop}

\begin{proof}
The finite index subgroup we will work with is the group $\FR(G) \rtimes \Fact(G)$, with the finite quotient being $\Perm(G)$. Observe that all the generators of this group preserve the conjugacy class of each free factor, making the normal closure of any collection of free factors `characteristic for this subgroup'.  Let $N$ be the normal closure of all but two factors. There is a map to $\Aut(G/N)$, and all generators (apart from the permutation, if present) are in the image. So we have a quotient isomorphic to (a finite index subgroup of) some $\Aut(H)$, where $H$ is a free product of just two groups. (If all the free factors are isomorphic, then $\Aut(H)$ necessarily contains a permutation automorphism, which is not in the image. However, its index 2 subgroup $\FR(H) \rtimes \Fact(H)$ works just as well for the rest of the argument.) By Proposition \ref{prop:two_factors} this admits an action on a tree and so does not have Property (FA).
\end{proof}

\propertyT

\begin{proof}
Discrete groups with Property (T) are finitely generated \cite[Theorem 1.3.1]{BekkaEtAlPropertyT}, so if any factor is uncountable they certainly do not have Property (T). If all factors are countable, we may use Watatani's result \cite{Watatani1982} which gives that if $\Aut(G)$ had Property (T), then every finite index subgroup would have Property (FA). Since Proposition \ref{prop:no_prop_T} gives a finite index subgroup which acts on a tree, and therefore does not have Property (FA), $\Aut(G)$ cannot have Property (T).
\end{proof}

\bibliographystyle{plain}

\appendix
\section{A presentation of \texorpdfstring{$\Out(G)$}{Out(G)}}
This appendix contains a proof of the presentation of $\Out(G)$ given in Section 4:
\outerauts*

The proof is largely the same as that given in \cite{CollinsGilbert1990AutFreeProd} for three non-isomorphic factors, differing by taking account of the permutation automorphisms which appear when the factor groups are isomorphic.

A presentation of $\Aut(G)$ (derived from Propositions \ref{prop:aut_G_pres} and \ref{prop:structure_of_factor_perm_auts}) consists of:

Generators: $(A,b),(A,c),(B,a),(B,c),(C,a),(C,b); \Aut(A)^3; \sigma_{(123)},\sigma_{(12)}$\\
\setcounter{equation}{0}
Relations (where $\varphi$ means an arbitrary factor automorphism, and $\sigma$ a permutation automorphism; and relations should be taken to range over all appropriate generators):\begin{align}
[(A,b),(C,b')]&=1 \text{,  etc.} \\
[(A,b)(C,b),(A,c)]&=1 \text{,  etc.} \\
(A,b)(A,b')&=(A,b'b) \text{,  etc.} \\
\varphi \varphi'&= \varphi'' \qquad \text{from the direct product structure} \\
\sigma_{(123)}^3=1,\sigma_{(12)}^2&=1, (\sigma_{(123)}\sigma_{(12)})^2=1 \qquad \text{(relations for $S_3$)} \\
\varphi^{-1} (A,b)\varphi&=(A, b\varphi) \text{,  etc.} \\
\sigma^{-1} (A,b)\sigma&=(A\sigma, b\sigma) \text{,  etc.} \\
\sigma^{-1}\varphi\sigma&=\varphi' \qquad\text{from the wreath product structure}
\end{align}

This gives $\Aut(G)$ as the iterated semidirect product $F\R(G)\rtimes \Aut(A)^3 \rtimes S_3$, where $\Aut(A)^3 \rtimes S_3$ is the permutation wreath product in Proposition \ref{prop:structure_of_factor_perm_auts}, and which can be evaluated in either order.

To find a presentation of $\Out(G)$, we add relations to this presentation setting each inner automorphism equal to the identity. That is, $\gamma(a)(B,a)(C,a)=1$ (where $\gamma(a)$ is the inner factor automorphism conjugating $A$ by $a \in A$ and fixing the other factor groups).

Use the new relation to rewrite three kinds of generators ($(A,c)$, $(B,a)$ and $(C,b)$) as (eg) $(A,c)=\gamma(c^{-1})(B,c^{-1})$. Then we can eliminate both those generators and the new relations. Putting this substitution in the first kind of relation we see that they are implied by the others (and so are unnecessary): \begin{align*}
[(A,c),(B,c')] &= [\gamma(c^{-1})(B,c^{-1}),(B,c')] \\ 
&= (B,c)\gamma(c)(B,c'^{-1})\gamma(c^{-1})(B,c^{-1})(B,c') \\
&= (B,c)(B,cc'^{-1}c^{-1})(B,c^{-1})(B,c') &\text{by (6)}\\
&= (B,c'c^{-1}cc'^{-1}c^{-1}c) &\text{by (3)}\\
&= (B,1) \\
&= 1
\end{align*}

Similarly for the second kind: \begin{align*}
[(A,b)(C,b),(A,c)] &= [\gamma(b^{-1}),\gamma(c^{-1})(B,c^{-1})] \\
&= \gamma(b)(B,c)\gamma(c)\gamma(b^{-1})\gamma(c^{-1})(B,c^{-1}) \\
&= \gamma(b)(B,c)\gamma(b^{-1})(B,c^{-1}) &\text{by (4)}\\
&= (B,c)(B,c^{-1}) &\text{by (6)} \\
&=1
\end{align*}

(There are also some versions of these than only require one substitution, say with $(C,a)$ instead of $(A,c)$.

Relations (3),(4),(5), and (8) are all in terms of only generators we eliminated (in which case they have also been eliminated) or of only generators we still have, so don't need any rewriting. Relations (6) only have the effect of changing the conjugating element for another drawn from the same factor group, so again don't require any rewriting.

However, (7) requires rewriting for transpositions. Taking $\sigma_{(12)}$, to interchange $A$ and $B$, we have \begin{align*}
\sigma_{(12)} (A,b)\sigma_{(12)}&=(A\sigma_{(12)}, b\sigma_{(12)})\\
&= (B,a) \\
&= \gamma(a^{-1})(C,a^{-1})
\end{align*}
And similarly: \begin{align*}
\sigma_{(12)}^{-1} (C,a)\sigma_{(12)}&= \gamma(b^{-1})(A,b^{-1}) \\
\sigma_{(12)}^{-1} (B,c)\sigma_{(12)}&= \gamma(c^{-1})(B,c^{-1})
\end{align*}

So with generators $\sigma_{(12)}$ (interchanging $A$ and $B$) and $\sigma_{(123)}$ (cycling $A$ to $B$ to $C$ to $A$) we replace (7) above with:
\setcounter{equation}{6}
\begin{subequations}
\begin{align}
\sigma_{(12)}^{-1} (A,b)\sigma_{(12)}&= \gamma(a^{-1})(C,a^{-1}) \\
\sigma_{(12)}^{-1} (B,c)\sigma_{(12)}&= \gamma(c^{-1})(B,c^{-1}) \\
\sigma_{(12)}^{-1} (C,a)\sigma_{(12)}&= \gamma(b^{-1})(A,b^{-1}) \\
\sigma_{(123)}^{-1} (A,b) \sigma_{(123)} &= (B,c) \text{,  etc.}
\end{align}
\end{subequations}

After eliminating (1) and (2), and replacing (7) by (7a)-(7d), this gives the presentation of Proposition \ref{prop:outerauts}.
\end{document}